\documentclass[a4paper]{article}

\usepackage[all]{xy}\usepackage[latin1]{inputenc}        
\usepackage[dvips]{graphics,graphicx}
\usepackage{amsfonts,amssymb,amsmath,color,mathrsfs, amstext}
\usepackage{amsbsy, amsopn, amscd, amsxtra, amsthm,authblk, enumerate}
\usepackage{upref}
\usepackage[colorlinks,
            linkcolor=red,
            anchorcolor=red,
            citecolor=red
            ]{hyperref}

\usepackage{geometry}
\usepackage[displaymath]{lineno}
\usepackage{float}

\usepackage{yhmath}

\usepackage[normalem]{ulem}

\numberwithin{equation}{section}

\DeclareMathOperator{\loc}{loc}

\newtheorem{theorem}{Theorem}[section]
\newtheorem{lemma}{Lemma}[section]
\newtheorem{proposition}{Proposition}[section]
\newtheorem*{proposition*}{Proposition}
\newtheorem{corollary}{Corollary}[section]
\newtheorem*{corollary*}{Corollary}
\newtheorem{definition}{Definition}[section]
\newtheorem*{definitions*}{Definitions}
\newtheorem*{conjecture*}{\bf Conjecture}

\theoremstyle{remark}
\newtheorem{remark}{\bf Remark}[section]

\begin{document}

\title{A note on one-dimensional time fractional ODEs}

\author[1]{Yuanyuan Feng\thanks{yuanyuaf@andrew.cmu.edu}}
\author[2]{Lei Li\thanks{leili@math.duke.edu}}
\author[3]{Jian-Guo Liu\thanks{jliu@phy.duke.edu}}
\author[4]{Xiaoqian Xu \thanks{xxu@math.cmu.edu}}
\affil[1,4]{Department of Mathematics, Carnegie Mellon University, Pittsburgh, PA 15213, USA}
\affil[2]{Department of Mathematics, Duke University, Durham, NC 27708, USA.}
\affil[3]{ Departments of Mathematics and Physics, Duke University, Durham, NC 27708, USA.}

\date{}
\maketitle

\begin{abstract}
In this note, we prove or re-prove several important results regarding one dimensional time fractional ODEs following our previous work \cite{fllx17}. Here we use the definition of Caputo derivative proposed in \cite{liliu17frac1,liliu2017} based on a convolution group.  In particular, we establish generalized comparison principles consistent with the new definition of Caputo derivatives. In addition, we establish the full asymptotic behaviors of the solutions for $D_c^{\gamma}u=Au^p$. Lastly, we provide a simplified proof for the strict monotonicity and stability in initial values for the time fractional differential equations with weak assumptions.
\end{abstract}

\section{Introduction}
The fractional calculus in time has been used widely in physics and engineering for memory effect, viscoelasticity, porous media etc \cite{gm97, kst06, diethelm10, ala16,liliu2017note}. There is a huge amount of literature discussing time fractional differential equations. For instance, one can find some results in \cite{df02, diethelm10} using the classic Caputo derivatives.  In this paper, we study the following time fractional ODE:
\begin{equation}\label{eq:fode1}
D_c^{\gamma}u= f(t, u), ~ u(0)=u_0,
\end{equation}
for $\gamma\in (0, 1)$  and $f$ measurable. Here $D_c^{\gamma}u$ is the generalized Caputo derivative introduced in \cite{liliu17frac1,liliu2017}. As we will see later, this generalized definition is theoretically more convenient, since it allows us to take advantage of the underlying group structure.

As in \cite{liliu17frac1}, we use the following distributions $\{g_\beta\}$ as convolution kernels for $\beta\in (-1, 0)$:
\begin{gather*}
g_{\beta}(t)=\frac{1}{\Gamma(1+\beta)}D\left(\theta(t)t^{\beta}\right).
\end{gather*} 
Here $\theta(t)$ is the standard Heaviside step function, $\Gamma(\cdot)$ is the gamma function, and $D$ means the distributional derivative on $\mathbb{R}$. Indeed, $g_{\beta}$ can be defined for $\beta\in \mathbb{R}$ (see \cite{liliu17frac1}) so that $\{g_{\beta}: \beta\in\mathbb{R}\}$ forms a convolution group. In particular, we have 
\begin{gather}\label{eq:group}
g_{\beta_1}*g_{\beta_2}=g_{\beta_1+\beta_2}.
\end{gather}
Here since the support of $g_{\beta_i}$ ($i=1,2$) is bounded from left, the convolution is well-defined.  Now we are able to give the generalized definition of fractional derivatives:
\begin{definition}[\cite{liliu17frac1,liliu2017}] \label{def:caputo}
Let $0<\gamma<1$. Consider $u\in L_{\loc}^1[0, T)$. Given $u_0\in \mathbb{R}$, we define the $\gamma$-th order generalized Caputo derivative of $u$, associated with initial value $u_0$, to be a distribution in $\mathscr{D}'(-\infty, T)$ with support in $[0, T)$, given by 
\[
D_c^{\gamma}u=g_{-\gamma}*\Big((u-u_0)\theta(t)\Big).
\]
If$\ \lim_{t\to 0+}\frac{1}{t}\int_0^t|u(s)-u_0|ds=0$, we call $D_c^{\gamma}u$ the Caputo derivative of $u$.
\end{definition}
As in \cite{liliu17frac1}, if the function $u$ is absolutely continuous, the generalized definition reduces to the classical definition. However, the generalized definition is theoretically useful because it reveals the underlying group structure (see Proposition \ref{pro:equi}).
\begin{definition}
Let $T>0$. A function $u\in L_{\loc}^1[0, T)$ is a weak solution to \eqref{eq:fode1} on $[0, T)$ with initial value $u_0$, if $f(t, u(t))\in \mathscr{D}'(-\infty,T)$ and the equality holds in the distributional sense. We call a weak solution $u$ a strong solution if  (i). $\lim_{t\to 0+}\frac{1}{t}\int_0^t|u(s)-u_0|ds=0$;  (ii). both $D_c^{\gamma}u$ and $f(t, u(t))$ are locally integrable on $[0,T)$.
\end{definition}
By the group property \eqref{eq:group}, we have
\begin{proposition}[\cite{liliu17frac1}]\label{pro:equi}
Suppose $f\in L_{\loc}^{\infty}([0,\infty)\times\mathbb{R}; \mathbb{R})$. Fix $T>0$. Then, $u(t)\in L_{\loc}^1[0, T)$ with initial value $u_0$ is a strong solution of \eqref{eq:fode1} on $(0, T)$ if and only if $ \lim_{t\to 0_+} \frac{1}{t} \int_0^t  |u(s)-u_0| \,ds=0$ and it solves the following integral equation 
\begin{gather}\label{eq:vol}
u(t)=u_0+\frac{1}{\Gamma(\gamma)}\int_0^t(t-s)^{\gamma-1}f(s, u(s))ds,~\forall t\in (0, T).
\end{gather}
\end{proposition}

Using this integral formulation, the following has been shown in \cite{liliu17frac1} 
\begin{proposition}\label{pro:exisunique}
Suppose $f:  [0,\infty)\times (\alpha, \beta)\to \mathbb{R}$ is continuous and locally Lipschitz continuous in $u$. For any given initial value $u_0\in (\alpha, \beta)$, there is a unique strong solution, which either exists globally on $[0,\infty)$ or approaches the boundary of $(\alpha, \beta)$ in finite time. Moreover, this solution is continuous on the interval of existence.
\end{proposition}

Below in Section \ref{sec:comparison}, we will establish some generalized comparison principles consistent with the new definition of Caputo derivatives. In Section \ref{asymp}, we establish the full asymptotic behaviors of the solutions for $D_c^{\gamma}u=Au^p$. In Section \ref{monotonicity}, we provide a new proof for the strict monotonicity and stability in initial values with weak assumptions.

\section{Generalized comparison principles}\label{sec:comparison}

The comparison principles are important in the analysis of time fractional PDEs (See \cite{liliuwang2017cauchy}). There are many versions of comparison principles proved in literature using various definitions of Caputo derivatives. In \cite{liliu17frac1},  the authors assumed $f(t,\cdot)$ to be non-decreasing. In \cite[Lemma 2.6]{vergara2015}, $f(t,\cdot)$ was assumed to be non-increasing.
In \cite[Theorem 2.3]{ramirez2012}, there is no assumption on the monotonicity of $f(t,\cdot)$, but the function $v$ is assumed to be $C^1$ so that the pointwise value of $D_c^{\gamma}v$ can be defined. Combining these ideas and establishing a crucial lemma (Lemma \ref{lmm:importantineq}), we prove some generalized comparison principles in this section.  Similar to \cite{liliu17frac1}, we define the inequality in the distributional sense:
\begin{definition}\label{def:inequalitydistribution}
Let $U$ be an open interval. We say $f\in\mathscr{D}'(U)$ is a nonpositive (nonnegative) distribution if for any $\varphi\in C_c^{\infty}(U)$ with $\varphi\ge0$, we have  $\langle f, \varphi\rangle\le 0~~ (\langle f, \varphi\rangle\ge 0)$. 
We say $f_1\le f_2$ in the distributional sense for $f_1, f_2\in \mathscr{D}'(U)$, if $f_1-f_2$ is nonpositive. We say $f_1\ge f_2$ in the distributional sense if $f_1-f_2$ is nonnegative.
\end{definition}

In order to prove the comparison principle, we first prove the following auxiliary lemma:
\begin{lemma}\label{lmm:importantineq}
Suppose $u\in L_{\loc}^1[0, T)$ and $\lim_{t\to 0^+}\frac{1}{t}\int_0^t|u(s)-u_0|\,ds=0$. If there exists a function  $f\in L_{\loc}^1(0, T)$ such that on interval $(0, T)$ we have in the distributional sense that
$D_c^{\gamma}u\le f$,
then for any given $A\in\mathbb{R}$, we have in the distributional sense
\[
D_c^{\gamma}(u-A)^+\le \chi(u\ge A) f,~~\text{on } (0,T).
\]
\end{lemma}
\begin{proof}
First, recall the following result in \cite [Proposition 3.11]{liliu17frac1}: if $u\in C[0, T)\cap C^1(0, T)$ and $u\mapsto E(u)$ is $C^1$ and convex, we have
\[
D_c^{\gamma}E(u)\le E'(u)D_c^{\gamma}u.
\]
Now let us consider $\eta\in C_c^{\infty}(-1, 0)$ with $\eta\ge 0$ and $\int \eta\,dt=1$. Define $\eta^{\epsilon}(t)=\frac{1}{\epsilon}\eta(\frac{t}{\epsilon})$ and $u^{\epsilon}=\eta^{\epsilon}*u$.
As showed in \cite[Proposition 3.11]{liliu17frac1},  $u^{\epsilon}(0)\to u_0$ and $u^{\epsilon}(t)\to u(t)$ in $L_{\loc}^1[0,T)$.

Denote $E(u)=(u-A)^+$ and define  $E^{\delta}(u)=(E*\eta^{\delta})(u)$. Clearly, $(E^{\delta})'(u)=\eta^{\delta}*\chi(u\ge A)$ is nonnegative and increasing,  which implies that $E^{\delta}$ is a convex increasing function. Then, we have
\begin{align}\label{star}
D_c^{\gamma}E^{\delta}(u^{\epsilon})\le (E^{\delta})'|_{u^{\epsilon}}D_c^{\gamma}u^{\epsilon}.
\end{align}
It is not hard to see  $\limsup_{\epsilon\to 0}(E^{\delta})'|_{u^{\epsilon}}D_c^{\gamma}u^{\epsilon} \le (E^{\delta})'|_{u} f(t)$. Since $E^{\delta}(u^{\epsilon})$ converges to $E^{\delta}(u)$ in $L_{\loc}^1$ and $E^{\delta}(u^{\epsilon}(0))$ converges to $E^{\delta}(u_0)$, according to Definition \ref{def:caputo}, $D_c^{\gamma}E^{\delta}(u^{\epsilon})\to D_c^{\gamma}E^{\delta}(u)$ as distributions. Moreover,  notice that the inequality is preserved in the distributional sense (Definition \ref{def:inequalitydistribution}). We have $D_c^{\gamma}E^{\delta}(u)\le  (E^{\delta})'|_{u} f(t)$. Taking $\delta\to 0$, similarly we have $D_c^{\gamma}E^{\delta}(u)$ converges as distributions to $D_c^{\gamma}(u-A)^+$. Then the right hand side of \eqref{star} converges to $\chi(u\ge A)f(t)$, and the inequality is preserved in the distributional sense.
\end{proof}

As is well-known, if $u\in H^1(0, T)$, $D(u-A)^+=\chi(u-A) Du$. Since Caputo derivative is nonlocal, the equality is no longer true in general. However, we have similar inequalities and Lemma \ref{lmm:importantineq} provides an answer.

\begin{corollary}\label{cor:capcomp}
Suppose $u(t)$ is a locally integrable function with $\lim_{t\to 0^+}\frac{1}{t}\int_0^t|u(s)-u_0|\,ds=0$. Let $A\in \mathbb{R}$ and $t_1\in (0, T)$ is a Lebesgue point. If $u\le A$ for a.e. $t\le t_1$, and on the interval $(t_1, T)$ we have $D_c^{\gamma}u\le 0$ in the distributional sense, then we have $u\le A, a.e.~(0, T)$.
\end{corollary}
Let $u^{\epsilon}$ be the mollification in the proof of  Lemma \ref{lmm:importantineq}. Consider $v^{\epsilon}=u^{\epsilon}-\frac{C(\epsilon) \theta(t)}{\Gamma(1+\gamma)}t^{\gamma}$ such that $v^{\epsilon}\le A$ for $t\in [0, t_1+\epsilon]$. $C(\epsilon)\to 0$ since $t_1$ is a Lebesgue point. Applying Lemma \ref{lmm:importantineq}, $D_c^{\gamma}(v^{\epsilon}-A)^+\le \chi(t\ge t_1+\epsilon)(D_c^{\gamma}u^{\epsilon}-C(\epsilon))\le \chi(t\ge t_1+\epsilon)(D_c^{\gamma}u^{\epsilon}-\eta_{\epsilon}*D_c^{\gamma}u)$.  Taking $\epsilon\to 0$ yields $D_c^{\gamma}(u-A)^+\le 0$. The details are left to readers. Now several versions of comparison principles can be stated as follows:
\begin{theorem}
\begin{enumerate}[(i)]
\item Suppose $u_i\in L_{\loc}^1[0, T)$ with $\lim_{t\to 0^+}\frac{1}{t}\int_0^t|u_i(s)-u_{i,0}|\,ds=0$ ($i=1,2$). Suppose $u_1(t)\le u_2(t)$ on $[0, t_1]$ for a Lebesgue point $t_1$, and  the $\gamma$-th Caputo derivatives of $u_1$, $u_2$ on $[0, t_1]$ are locally integrable. Define
\[
h_i(t)=u_{i,0}+\frac{1}{\Gamma(\gamma)}\int_0^{t\wedge t_1}(t-s)^{\gamma-1}D_c^{\gamma}u_i(s)\,ds,~i=1,2.
\]
Then, $h_1(t)\le h_2(t)$ for all $t\in [0, T]$. Moreover,  assume there exists a measurable function $f(t,u)$ such that (i) $f(\cdot, u_i(\cdot))$ ($i=1,2$) is locally integrable on $[t_1,T)$; (ii)  $f(t,\cdot)$ is non-decreasing on $[t_1, T)$; (iii) $D_c^{\gamma}u_1\le f(t, u_1)$ and $D_c^{\gamma}u_2\ge f(t, u_2)$ in the distributional sense on $(t_1, T)$, then $u_1\le u_2$ a.e. on $[0, T)$.

\item Suppose $u_i\in L_{\loc}^1[0, T)$ with $\lim_{t\to 0^+}\frac{1}{t}\int_0^t|u_i(s)-u_{i,0}|\,ds=0$ ($i=1,2$). If $u_1\le u_2$ on $[0, t_1]$ for a Lebesgue point $t_1$ and $D_c^{\gamma}(u_1-u_2)\le f(t, u_1)-f(t, u_2)$ as distributions on $(t_1, T)$, with $f(t,\cdot)$ being non-increasing on $(t_1, T)$ and $f(\cdot, u_i(\cdot))$ ($i=1,2$) being locally integrable on $[t_1, T)$, then $u_1\le u_2$ a.e on $[0, T)$.

\item Suppose $u(t)$ is a continuous function on $[0, T]$. If $u(t_1)=\sup_{0\le s\le t_1}u(s)$ for some $t_1\in (0, T]$ and $f(t)=D_c^{\gamma}u(t)$ is a continuous function, then $f(t_1)\ge 0$.
\end{enumerate}
\end{theorem}
\begin{proof}
(i). Clearly, $D_c^{\gamma}h_i=D_c^{\gamma}u_i$ for $t\le t_1$ and $D_c^{\gamma}h_i=0$ for $t>t_1$. Let $u=h_1-h_2$, $A=0$ in Corollary \ref{cor:capcomp}, we find $h_1\le h_2$.
On $[t_1, T)$, we have
\[
u_1(t)\le h_1(t)+\frac{1}{\Gamma(\gamma)}\int_{t_1}^{t}(t-s)^{\gamma-1}f(s, u_1)\,ds,~~~~ u_2(t)\ge h_2(t)+\frac{1}{\Gamma(\gamma)}\int_{t_1}^{t}(t-s)^{\gamma-1}f(s, u_2)\,ds.
\]
As $h_1(t)\le h_2(t)$ and $f(t,\cdot)$ is non-decreasing, one has $u_1(t)\le u_2(t)$ (see \cite[Theorem 4.10]{liliu17frac1}).

(ii). Apply Lemma \ref{lmm:importantineq} for $u_1-u_2$ and $A=0$. (The proof is similar as in Corollary \ref{cor:capcomp}.)

(iii). Consider $u^{\epsilon}(t)=u(t)+\frac{\epsilon \theta(t)}{\Gamma(1+\gamma)} t^{\gamma}$, where $\epsilon>0$. Then, $t_1$ is the unique maximizer of $u^{\epsilon}$ on $[0,t_1]$. Let $f^{\epsilon}=D_c^{\gamma}u^{\epsilon}=f+\epsilon$. It suffices to show
\begin{gather}\label{eq:fepsilon}
f^{\epsilon}(t_1)\ge 0,~\forall \epsilon>0.
\end{gather}
Otherwise, there is an $\epsilon_0>0$ such that $f^{\epsilon_0}(t_1)<0$. Since $f^{\epsilon_0}$ is continuous, we can find $\delta>0$ such that on $[t_1-\delta, t_1]$ $f^{\epsilon_0}$ is negative and $u^{\epsilon_0}(t)\le u^{\epsilon_0}(t_1-\delta)$ for $t\le t_1-\delta$. Applying Corollary \ref{cor:capcomp}, we have $u^{\epsilon_0}(t)\le u^{\epsilon_0}(t_1-\delta)$ for $t\in [t_1-\delta, t_1]$, which is a contradiction. Taking $\epsilon\to 0$ then gives the result.
\end{proof}
\begin{remark}
Though the conditions here are weaker under the new definition of Caputo derivative, (ii) is essentially \cite[Lemma 2.6]{vergara2015} and  (iii) is well-known for $C^1$ functions (see, for example \cite{luchko2009,ramirez2012}).
\end{remark}

Now, we establish a generalized Gr\"onwall inequality (or another version of comparison principle), consistent with the new definition of Caputo derivative. The main construction is inspired by \cite{ramirez2012}.
\begin{theorem}\label{compar}
Suppose $f(t,u)$ is continuous and locally Lipschitz in $u$. Let $v(t)$ be a continuous function.
If $D_c^{\gamma}v\le f(t, v)$ in the distributional sense, and $D_c^{\gamma}u=f(t, u)$, with $v_0\le u_0$. Then, $v\le u$ on the common interval. Similarly, if we have $D_c^{\gamma}v\ge f(t, v)$ as distributions and $v_0\ge u_0$, then $v\ge u$ on the common interval.
\end{theorem}

\begin{proof}
We only prove the first claim (the proof for the other is similar). By Proposition \ref{pro:exisunique},  $D_c^{\gamma}u=f(t, u)$ with initial value $u(0)=u_0$ has a unique solution on the interval $[0, T_b)$, where $T_b$ is the largest time of existence. Moreover, $u$ is continuous on $[0, T_b)$.

Fix $T\in (0, T_b)$. Pick $M$ large enough so that $u(t)$ and $v(t)$ fall into $[0, T]\times [-M, M]$. Let $L$ be the Lipschitz constant of $f(t,\cdot)$ for the region $[0, T]\times [-2M, 2M]$. Consider
\[
v^{\epsilon}=v-\epsilon w.
\]
Here $w=E_{\gamma}(2L t^{\gamma})$ is the solution to $D_c^{\gamma}w=2Lw$ with initial value $1$, where $E_{\gamma}(z)=\sum_{n=0}^\infty\frac{z^n}{\Gamma(n\gamma+1)}$ is the Mittag-Leffler function \cite{hms11,mg00}. Clearly, if $\epsilon$ is sufficiently small, $v^{\epsilon}$ falls into $[0, T]\times [-2M, 2M]$.  Then, we find that in the distributional sense
\[
D_c^{\gamma}v^{\epsilon}=D_c^{\gamma}v-\epsilon 2L w
\le f(t,v)-\epsilon 2Lw\le f(t, v^{\epsilon})-\epsilon Lw.
\]
We claim that for all such small $\epsilon$,
\begin{gather}\label{eq:vperturbed}
v^{\epsilon}(t)\le u(t), \forall t\in [0, T].
\end{gather} 
 If not, define
 \[
 t_1=\sup\{t\in (0, T]: v^{\epsilon}(s)\le u(s),~\forall s\in [0, t] \}.
 \]
 Since $v^{\epsilon}(0)=v_0-\epsilon<u_0$, by continuity  we have $t_1>0$. By assumption, \eqref{eq:vperturbed} is not true, and we have $t_1<T$.  Consequently, there exists $\delta_1>0$, such that 
$v^{\epsilon}(t_1)=u(t_1)$  and $v^{\epsilon}(t)>u(t)$ for $t\in (t_1, t_1+ \delta_1)$.
 Moreover, 
\[
D_c^{\gamma}(v^{\epsilon}-u)\le f(t, v^{\epsilon})-\epsilon Lw-f(t,u).
\]
By continuity, for some $\delta_2\in (0,\delta_1)$, $D_c^{\gamma}(v^{\epsilon}-u)$ is a nonpositive distribution on the interval $(t_1, t_1+\delta_2)$. By Corollary \ref{cor:capcomp}, we have $v^{\epsilon}(t)\le u(t)$ for $t\in (t_1, t_1+\delta_2)$, which is a contradiction.  Hence, \eqref{eq:vperturbed} is true.
Taking $\epsilon\to 0$ in \eqref{eq:vperturbed} yields the result on $[0, T]$. Since $T$ is arbitrary, the result is true.
\end{proof}

\section{Asymptotic behaviors for a class of fractional ODEs}\label{asymp}
In this section, we study the solution curves to the following autonomous fractional ODEs:
\begin{gather}\label{eq:fracode2}
D_c^{\gamma}u=Au^p,~~u(0)=u_0>0.
\end{gather}
The monotonicity of the solutions to \eqref{eq:fracode2} and some partial results for the asymptotic behaviors have been established in our previous work \cite{fllx17}.  The asymptotic behaviors of the solutions for the $A<0, p>0$ case have also been discussed in \cite[Theorem 7.1]{vergara2015}.
However, the discussion on all the range of $A$ and $p$ is not complete. Here, we will give a complete description on asymptotic behaviors of the solution curves.

By Proposition \ref{pro:exisunique}, the strong solution $u$ to \eqref{eq:fracode2} exists on $[0, T_b)$ for $T_b\in (0,\infty]$. If $T_b<\infty$, either $\lim_{t\to T_b^-}u(t)=0$ or $\lim_{t\to T_b^-}u(t)=\infty$. We give a complete description regarding the solutions curves to \eqref{eq:fracode2}:
\begin{theorem}
Consider \eqref{eq:fracode2}. If $A=0$, then $u(t)=u_0$. If $A>0$,  then all the solutions are strictly increasing on $(0, T_b)$. If $A<0$, then all solutions are strictly decreasing before they touch $0$. 
\begin{enumerate}[(i)]
\item Suppose $A>0$. If $p>1$, then $T_b<\infty$ and $u(t)\sim  \left[\frac{\Gamma(\frac{p\gamma}{p-1})}{A\Gamma(\frac{\gamma}{p-1})}\right]^{\frac{1}{p-1}}(T_b-t)^{-\frac{\gamma}{p-1}}$, as $t\to T_b^-$. If $p=1$, then $u(t)=u_0E_{\gamma}(At^{\gamma})$. If $p<1$, then there exist $c_1>0$ and $c_2>0$ such that $
c_1 t^{\frac{\gamma}{1-p}}\leq u(t)\leq c_2t^{\frac{\gamma}{1-p}}, ~ t\geq 1$.
\item Suppose $A<0$. If $p<0$, the solution curve touches $u=0$ in finite time where the right hand side blows up. If $p=0$, then $u=u_0+Ag_{1+\gamma}$. If $p>0$, then $T_b=\infty$, and there exist $c_1>0, c_2>0$ such that 
$c_1t^{-\frac{\gamma}{p}}\le u(t) \le c_2t^{-\frac{\gamma}{p}}, ~ t\ge 1$.
\end{enumerate} 
\end{theorem}
\begin{proof}
The $A=0$ or $p=0$ cases are trivial. The monotonicity has been proved in \cite{fllx17}. The $A>0, p>1$ case has also been discussed there. Indeed, there is also an accurate estimate of $T_b$ in  \cite{fllx17}. The $p=1$ case is trivial.  The $A<0, p>0$ case has been discussed in \cite[Theorem 7.1]{vergara2015}. In fact, they established a version of comparison principle and used a subsolution and a supersolution to get $c_1t^{-\frac{\gamma}{p}}\le u(t) \le c_2t^{-\frac{\gamma}{p}}, \quad t\ge 1$.
For the case $A<0, p<0$, since the solution is decreasing, we have $D_c^{\gamma}u\le Au_0^p<0$ before $u$ touches zero. Hence, the claim follows.

Now, we establish the results for $A>0, p<1$ case. First, let us construct the sub-solution as follows:
\[
\omega(t)=
\begin{cases}
u_0, & t\in [0,t_0],\\
a t^{\frac{\gamma}{1-p}},& t\geq t_0.
\end{cases}
\]
Here $a>0$ is to be determined and $t_0$ is determined by $at_0^{\frac{\gamma}{1-p}}=u_0$. Clearly, $\omega$ is absolutely continuous on any finite interval. For $t<t_0$, $D_c^{\gamma}\omega=0\leq A\omega^{p}$. For $t\geq t_0$, we have
\[
D_c^{\gamma}\omega= \frac{a\gamma}{(1-p)\Gamma(1-\gamma)}\int_{t_0}^{t}\frac{\tau^{\frac{\gamma}{1-p}-1}}{(t-\tau)^{\gamma}}d\tau
< \frac{a \gamma B(\frac{\gamma}{1-p}, 1-\gamma)}{(1-p)\Gamma(1-\gamma)}t^{\frac{\gamma p}{1-p}} =
\frac{a\Gamma(\gamma/(1-p)+1)}{\Gamma(\gamma p/(1-p)+1)}t^{\frac{p\gamma}{1-p}},
\]
where $B(\cdot,\cdot)$ is the Beta function. Clearly, if we choose $a>0$ such that $\frac{a\Gamma(\gamma/(1-p)+1)}{\Gamma(\gamma p/(1-p)+1)}\le Aa^p$, then $D_c^{\gamma}\omega\le A\omega^p$. Such $a$ exists because $p<1$.

For the super-solution, let us consider
\[
v(t)=
\begin{cases}
u_0+B_1\frac{t^{\gamma}}{\Gamma(1+\gamma)},& t\in [0,1],\\
B_2 t^{\frac{\gamma}{1-p}}, & t\geq 1.
\end{cases}
\]
$B_2$ is determined by $B_2=u_0+\frac{B_1}{\Gamma(1+\gamma)}$. This choice of $B_2$ makes $v$ absolutely continuous on any finite interval. We now determine $B_1$. On $[0,1]$,  one has $D_c^{\gamma}v=B_1$.
For $t>1$, we have
\[
D_c^{\gamma}v=\frac{B_1\gamma}{B(1+\gamma,1-\gamma)}
\int_0^{1}\frac{\tau^{\gamma-1}}{(t-\tau)^{\gamma}}\,d\tau+
 \frac{B_2}{\Gamma(1-\gamma)}\frac{\gamma}{1-p}\int_{1}^{t}\frac{\tau^{\frac{\gamma}{1-p}-1}}{(t-\tau)^{\gamma}}d\tau.
\]
On $[1, 2]$, one has $D_c^{\gamma}v 
>\frac{B_1\gamma}{B(1+\gamma,1-\gamma)}\int_0^{1}\frac{\tau^{\gamma-1}}{(2-\tau)^{\gamma}}\,d\tau=B_1 C_1(\gamma)$.
For $t>2$, we have
\[
D_c^{\gamma}v>B_2 \frac{1}{\Gamma(1-\gamma)}\frac{\gamma}{1-p}t^{\frac{\gamma p}{1-p}}\int_{\frac{1}{t}}^1\frac{\tau^{\frac{\gamma}{1-p}-1}}{(1-\tau)^{\gamma}}d\tau\\
\geq B_2 t^{\frac{\gamma p}{1-p}} C_2(p,\gamma).
\]
It is clear that there exists $M_1(A, p, \gamma)$ such that as long
as $B_2\ge M_1$, $D_c^{\gamma}v\ge Av^p$ for $t\ge 2$ since $p<1$. For $v$ to be a super-solution, one needs 
\begin{gather*}
u_0+B_1\frac{1}{\Gamma(1+\gamma)}\ge M_1,~~~~~
B_1\min(1, C_1(\gamma))
\ge A \max\left(u_0^p, \Big(u_0+\frac{B_1}{\Gamma(1+\gamma)}\Big)^p 2^{\frac{p\gamma}{1-p}}\right).
\end{gather*}
Such $B_1$ exists since $p<1$. Hence, applying comparison principle Theorem \ref{compar} yields the result.
\end{proof}

\section{Strict monotonicity and stability in initial values}\label{monotonicity}

It is well-known that solution curves for well-behaved ODEs do not touch each other. However, for fractional ODEs, similar results are not trivial since the dynamics is non-Markovian. By the comparison principles (or generalized Gr\"onwall inequality), if $f(t,u)$ in \eqref{eq:fode1} is continuous and locally Lipschitz in $u$, $u(0)<v(0)$ implies $u(t)\le v(t)$ for $t\ge 0$. However we do not have strict inequality. In \cite[Theorem 6.12]{diethelm10}, the strict inequality has been established following a series of contraction techniques. Using our new definition of Caputo derivative, we provide a new proof of that solutions are strict monotone in initial values, by assuming $f\in L_{\loc}^{\infty}$. 

The following lemma (a variant of \cite[Lemma 3.4]{fllx17} or \cite[Theorem 1]{weis75}), is important:
\begin{lemma}\label{lmm:pos}
Let $r_{\lambda}(t)=-\frac{d}{dt}E_{\gamma}(-\lambda\Gamma(\gamma)t^{\gamma})$ be the resolvent for kernel $\lambda t^{\gamma-1}$ (in other words, $r_{\lambda}(t)+\lambda\int_0^t(t-s)^{\gamma-1}r_{\lambda}(s)ds=\lambda t^{\gamma-1}$).
Let $T>0$. Assume $h\in L^1[0,T]$, $h>0~a.e.$, satisfying 
\[
h(t)-\int_0^tr_{\lambda}(t-s) h(s)ds>0, ~a.e., \quad\forall \lambda>0.
\]

Suppose $v\in L^{\infty}[0,T]$, then the integral equation 
\begin{gather}\label{eq:varcoevol}
y(t)+\int_0^t(t-s)^{\gamma-1}v(s) y(s)ds=h(t)
\end{gather}
has a unique solution  $y(t)\in L^1[0, T]$. Moreover, $
y(t)>0, a.e..
$
\end{lemma}
The proof is exactly the same as \cite[Lemma 3.4]{fllx17}, though we only assume $v\in L^{\infty}[0, T]$ here.
Next, we provide a new proof for the strict monotonicity in initial value. We also prove the stability of solutions with respect to initial values.
\begin{theorem}\label{lmm:perturbinitial}
Assume that $f(\cdot,\cdot)\in L_{\loc}^{\infty}([0,\infty)\times\mathbb{R})$. Moreover, assume for every compact set $K$, there is $L_K>0$ such that $|f(t,u)-f(t,v)|\le L_K|u-v|$ for a.e. $(t,u), (t,v)\in K$.  Then, for a given initial value $u_0$,  the solution in $L_{\loc}^{\infty}[0, T_b)$ is unique.  Further, we have
\begin{itemize}
\item Any two solutions $u_i\in L_{\loc}^{\infty}[0, T_b^i)$ $(i=1,2)$ with initial values $u_{1,0}<u_{2,0}$ satisfy $u_1(t)<u_2(t)$ on $[0, \min(T_b^1, T_b^2))$. 

\item For any $T>0$, $M>0$, there exists $C(M,T)>0$ such that any two solutions with $\|u_i\|_{L^{\infty}[0, T]}\le M$ ($i=1,2$) and initial values $u_{1,0}, u_{2,0}$ satisfy
 \[
 \|u_1-u_2\|_{L^{\infty}[0,T]}\le C(M,T)|u_{1,0}-u_{2,0}|.
 \]
 \end{itemize}
\end{theorem}
\begin{proof}
Fix $T\in (0, \min(T_b^1, T_b^2))$. There exists $K$ compact such that for a.e $t\in [0, T]$, $(t, u_i(t))\in K$. By Proposition \ref{pro:equi}, one has
\[
u_i(t)=u_{i,0}+\frac{1}{\Gamma(\gamma)}\int_0^t(t-s)^{\gamma-1}f(s, u_i(s))\,ds.
\]
The boundedness of $f(s, u_i(s))$ implies that $u_i(t)\in C[0, T]$. If $u_{1,0}=u_{2,0}$, by taking the difference, $|u_1(t)-u_2(t)|\le C\int_0^t(t-s)^{\gamma-1}|u_1(s)-u_2(s)|\,ds$ and the uniqueness therefore follows.

 Now, assume $u_{1,0}\neq u_{2,0}$. Define $y(t)=(u_2(t)-u_1(t))/(u_{2,0}-u_{1,0})$, we  have
\[
y(t)+\int_0^t(t-s)^{\gamma-1}v(s)y(s)\,ds =1, \text{ where }v(s)=-\frac{1}{\Gamma(\gamma)} \frac{f(s,u_2(s))-f(s, u_1(s))}{u_{2}(s)-u_{1}(s)}.
\]
If $u_1(s)=u_2(s)$, we define $v(s)=0$. Note that $|v|\le L_K/\Gamma(\gamma)$ a.e. for $t\in (0, T)$. By setting $h=1$ in Lemma  \ref{lmm:pos}, one has
\[
1-\int_0^t r_{\lambda}(t-s)\,ds=E_{\gamma}(-\lambda \Gamma(\gamma)t^{\gamma})>0.
\]
By Lemma \ref{lmm:pos}, $y(t)>0$. Since $y$ is continuous{\color{red},} satisfying 
\[
y(t)\le 1+\int_0^t(t-s)^{\gamma-1}\|v\|_{L^{\infty}[0,T]}y(s)\,ds,
\]
we have $y(t)\le C(\|v\|_{L^{\infty}}, T)$ by \cite[Proposition 5]{fllx17}. This verifies the last claim.
\end{proof}

\section*{Acknowledgements}
The work of J.-G Liu was partially supported by KI-Net NSF RNMS11-07444 and NSF DMS-1514826. Y. Feng was supported by NSF DMS-1252912.

\bibliographystyle{plain}
\bibliography{nonlinearFODE}

\end{document}